\documentclass[12pt]{amsart}
\usepackage{amssymb}
\theoremstyle{plain}
 \newtheorem{thm}{Theorem}[section]
 \newtheorem{cor}[thm]{Corollary}
 \newtheorem{lem}[thm]{Lemma}
 \newtheorem{prop}[thm]{Proposition}

\theoremstyle{definition}

\theoremstyle{remark}

\begin{document}
\title[on the conjecture of the norm Schwarz inequality.]
{on the conjecture of the norm Schwarz inequality.}

\author[Tomohiro Hayashi]{{Tomohiro Hayashi} }
\address[Tomohiro Hayashi]
{Nagoya Institute of Technology, 
Gokiso-cho, Showa-ku, Nagoya, Aichi, 466-8555, Japan}

\email[Tomohiro Hayashi]{hayashi.tomohiro@nitech.ac.jp}

\baselineskip=17pt

\maketitle

\begin{abstract}
For any positive invertible matrix $A$ and 
any normal matrix $B$ in $M_{n}({\Bbb C})$, 
we investigate whether the inequality 
$
||A\sharp (B^{*}A^{-1}B)||\geq ||B||
$ is 
true or not, 
where $\sharp$ denotes the geometric mean and 
$||\cdot||$ denotes the operator norm. 
We will solve this problem negatively. 
The related topics are also discussed.
\end{abstract}

\subjclass
{}

\keywords
{}

\section{Introduction}
In the paper \cite{A2} Ando considered the following problem. 
For three matrices $A,B,C$ with $A\geq0,C\geq0$, 
does
$
\begin{pmatrix}
A&B\\
B^{*}&C
\end{pmatrix}\geq0
$ imply 
$
||A\sharp C||\geq||B||
$? Here $A\sharp C$ is the geometric mean 
of $A$ and $C$. 
The inequality 
$
||A\sharp C||\geq||B||
$ was called 
the norm Schwarz inequality. 
In the case that $A$ is invertible, 
it is known that 
$
\begin{pmatrix}
A&B\\
B^{*}&C
\end{pmatrix}\geq0
$ if and only if 
$
C\geq B^{*}A^{-1}B
$. So the above problem is equivalent to the following. 
Is 
$
||A\sharp (B^{*}A^{-1}B)||\geq ||B||
$ 
always true for $A>0$? 
Ando showed that if $B$ satisfies this inequality 
for any $A$, then 
$B$ must be normaloid (i.e., $||B||=r(B)$ 
the maximum of eigenvalues of $B$). 
Then it is natural to expect that this norm 
inequality 
holds whenever $B$ is normal. 

\ \\

\noindent\underline{\bf Conjecture.} 
For any positive invertible matrix $A$ and 
any normal matrix $B$ in $M_{n}({\Bbb C})$, we have 
$$
||A\sharp (B^{*}A^{-1}B)||\geq ||B||.
$$

\ \\

Ando showed the following \cite{A2}. 
\begin{enumerate}
\item If $B$ is normaloid, the inequality 
$
||A^{\frac{1}{2}}(B^{*}A^{-1}B)^{\frac{1}{2}}||\geq ||B||
$ holds. 
\item 
If $B$ is self-adjoint, the conjecture is true. 
\item 
If $B$ is a scalar multiple of a unitary matrix, 
the conjecture is true. 
\item When $n=2$, the conjecture is true. 
\end{enumerate}

The aim of this paper is to construct a counter-example 
to this conjecture in $M_{6}({\Bbb C})$. 
For this purpose, we introduce some statements which are equivalent 
to the above conjecture. As a bonus,  we can show that 
if the above conjecture is true, then the 
inequality 
$$
A\sharp B^{-1}+B\sharp C^{-1}+C\sharp A^{-1}\geq3I,
$$ 
must hold for any positive invertible matrices 
$A$, $B$ and $C$. Then we can construct a counter-example for this inequality. 
The idea of constructing a counter-example 
for this inequality 
is basically due to M. Lin, who
attributed it to Drury \cite{DL}. 
In the final section 
we can give another proof of Ando's theorem for $2\times2$ matrices. 

After finishing this work  the author learned from Professor Minghua Lin that 
he succeeded in constructing a counter example to the above conjecture 
before us. His example consists of $3\times3$ matrices and so 
it is better than ours. The idea of construction is different. 

The author wishes to express his hearty gratitude to Professor Tsuyoshi Ando
for valuable comments. The author is also grateful to Professors Minghua Lin 
and Stephen Drury. The example in section 3 is due to them. 
The author thanks Professors 
Yoshihiro Nakamura, Muneo Cho and Takeaki Yamazaki 
for their comments. 

\section{Some equivalent conjectures}
Throughout this paper we denote by $M_{n}({\Bbb C})$ 
the space of $n\times n$ matrices. 
The geometric mean of two positive matrices $A,B\in M_{n}({\Bbb C})$ 
is denoted by $A\sharp B$. 
If they are invertible, we can write 
$
A\sharp B=
A^{\frac{1}{2}}
(A^{-\frac{1}{2}}BA^{-\frac{1}{2}})^{\frac{1}{2}}
A^{\frac{1}{2}}
$. For a matrix $A$ we denote 
its trace and determinant by 
${\rm Tr}(A)$ and ${\rm det}(A)$ respectively. 
We also denote the operator norm of a matrix $A$ by $||A||$.

First we introduce three conjectures: 

\ \\

\noindent\underline{\bf Conjecture 1.} 
(Ando, \cite{A2}) 
For any positive invertible matrix $A$ and 
any normal invertible matrix $B$ in $M_{n}({\Bbb C})$, 
we have 
$$
||A\sharp (B^{*}A^{-1}B)||\geq ||B||.
$$

\ \\

\noindent\underline{\bf Conjecture 2. }
For any positive invertible matrix $S$, 
any unitary matrix $U$ and any 
positive invertible matrix $D$ in $M_{n}({\Bbb C})$ 
with $UD=DU$, we have 
$$
||D^{\frac{1}{2}}\cdot S\sharp (U^{*}S^{-1}U)\cdot D^{\frac{1}{2}}||
\geq ||D||.
$$

\ \\

For a unitary matrix $U$ with the spectral decomposition 
$U=\displaystyle{\sum_{i}}z_{i}P_{i}$ 
($z_{i}\not=z_{j}$, $\{P_{i}\}_{i}$ are spectral projections), 
we set 
$$
E_{U}(X)=\displaystyle{\sum_{i}}P_{i}XP_{i}.
$$ 
With respect to the Hilbert-Schmidt inner product 
$\langle X|Y\rangle={\rm Tr}(X^{*}Y)$ on $M_{n}({\Bbb C})$, 
the map $E_{U}(\cdot)$ is 
the orthogonal projection to the commutant of U, that is, to
the class $\{X;XU = UX\}$. 
$E_{U}(\cdot)$ is a unital, 
trace-preserving, positive (hence contractive) linear map on 
$M_{n}({\Bbb C})$ such that 
$E_{U}(DX) = D\cdot E_{U}(X)$, 
$E_{U}(XD) = E_{U}(X) \cdot D$ for any $D\geq 0$ 
with $DU = UD$.

Here we remark that if $U^{k}=I$ for some positive integer $k$, 
the map $E_{U}$ can also be 
defined by 
$$
E_{U}(X)=\dfrac{1}{k}\displaystyle{\sum_{i=0}^{k-1}}{U^{*}}^{i}XU^{i}.
$$

\ \\

\noindent\underline{\bf Conjecture 3.}
For any positive invertible matrix $S$ and 
any unitary matrix $U$ in $M_{n}({\Bbb C})$, we have 
$$
E_{U}(S\sharp (U^{*}S^{-1}U))\geq I.
$$

\ \\

The main result in this section is 
the following.

\begin{thm}
All three conjectures above 
are mutually equivalent. 
\end{thm}

\begin{proof}
\noindent(Conjecture 1 $\Rightarrow$ Conjecture 2) 
We set $B=UD=DU$ and $A=D^{\frac{1}{2}}SD^{\frac{1}{2}}$. 
Then we see that 
$$
A\sharp (B^{*}A^{-1}B)=(D^{\frac{1}{2}}SD^{\frac{1}{2}})
\sharp 
(D^{\frac{1}{2}}U^{*}S^{-1}UD^{\frac{1}{2}})
=D^{\frac{1}{2}}\cdot S\sharp (U^{*}S^{-1}U)\cdot 
D^{\frac{1}{2}}.
$$
Since $B$ is normal, applying Conjecture 1 we have 
$$
||D^{\frac{1}{2}}\cdot S\sharp (U^{*}S^{-1}U)
\cdot D^{\frac{1}{2}}||=
||A\sharp (B^{*}A^{-1}B)||\geq ||B||=||D||.
$$

\noindent(Conjecture 2 $\Rightarrow$ Conjecture 1) 
Take a polar decomposition $B=UD=DU$ 
with unitary $U$ and positive $D$ 
and set 
$S=D^{-\frac{1}{2}}AD^{-\frac{1}{2}}$. Then as shown 
above we have 
$
A\sharp (B^{*}A^{-1}B)
=D^{\frac{1}{2}}\cdot S\sharp (U^{*}S^{-1}U)
\cdot D^{\frac{1}{2}}
$ and hence Conjecture 2 implies Conjecture 1. 

\noindent(Conjecture 2 $\Rightarrow $Conjecture 3) 
It is enough to show that $e\cdot S\sharp (U^{*}S^{-1}U)\cdot e\geq e$ for 
any rank one projection $e$ with $Ue=eU$. 
Indeed, if $U$ has the spectral decomposition 
$U=\displaystyle{\sum_{i}}z_{i}P_{i}$ ($z_{i}\not=z_{j}$), then 
we can write 
$
E_{U}(X)=\displaystyle{\sum_{i}}P_{i}XP_{i}.
$ In order to show Conjecture 3, we have to show 
$P_{i}\cdot S\sharp (U^{*}S^{-1}U)\cdot P_{i}\geq P_{i}$ 
for each $i$. To do so, it is enough to show 
$e\cdot S\sharp (U^{*}S^{-1}U)\cdot e\geq e$ for 
any rank one projection $e\leq P_{i}$. 
Here we remark that a rank one projection $e$ satisfies $Ue=eU$ 
if and only if $e\leq P_{i}$ for some $i$. 

We set $D=e+\dfrac{1}{2}(I-e)$. Then by Conjecture 2 we have 
$$
||D^{\frac{n}{2}}\cdot S\sharp (U^{*}S^{-1}U)
\cdot D^{\frac{n}{2}}||\geq ||D^{n}||
$$ 
for any positive integer $n$. By tending $n\rightarrow \infty$ we have 
$$
||e\cdot S\sharp (U^{*}S^{-1}U)
\cdot e||\geq ||e||=1.
$$ 
Then since $e$ is a rank one projection, 
we conclude that 
$$
e\cdot S\sharp (U^{*}S^{-1}U)
\cdot e=
||e\cdot S\sharp (U^{*}S^{-1}U)
\cdot e||e\geq e.
$$

\noindent(Conjecture3$\Rightarrow$Conjecture2) 
We may assume $||D||=1$. Take a spectral projection $P$ of $D$ with 
$DP=P$. 
Notice that $P$ commutes with $U$. 
Then by Conjecture 3 we compute 
\begin{align*}
||D^{\frac{1}{2}}\cdot S\sharp (U^{*}S^{-1}U)
\cdot D^{\frac{1}{2}}||
&\geq 
||E_{U}(D^{\frac{1}{2}}\cdot S\sharp (U^{*}S^{-1}U)
\cdot D^{\frac{1}{2}})||\\
&\geq 
||P\cdot 
E_{U}(D^{\frac{1}{2}}\cdot S\sharp (U^{*}S^{-1}U)
\cdot D^{\frac{1}{2}})
\cdot P||\\
&=||PD^{\frac{1}{2}}\cdot 
E_{U}(S\sharp (U^{*}S^{-1}U))\cdot D^{\frac{1}{2}}P||\\
&=||P\cdot E_{U}(S\sharp U^{*}S^{-1}U)
\cdot P||\geq ||P||=1=||D||.
\end{align*}

\end{proof}

\begin{cor}
If Conjecture 1 is true in $M_{3n}({\Bbb C})$, then for any 
positive invertible matrices $A,B,C\in M_{n}({\Bbb C})$, 
we have 
$$
A\sharp B^{-1}+B\sharp C^{-1}+C\sharp A^{-1}\geq3I.
$$
\end{cor}

\begin{proof} 
Denote by $M_{3}(M_{n}({\Bbb C}))$ the space of 
$3\times 3$ matrices with entries $M_{n}({\Bbb C})$. 
It is canonically identified with $M_{3n}({\Bbb C})$. 
We set 
$U=
\begin{pmatrix}
0&0&I_{n}\\
I_{n}&0&0\\
0&I_{n}&0
\end{pmatrix}
$ and 
$
S=
\begin{pmatrix}
A&0&0\\
0&B&0\\
0&0&C
\end{pmatrix}
$. By the previous theorem Conjecture 3 is also true. 
We will apply Conjecture 3 to these matrices. 

It is easy to see that 
$$
S\sharp (U^{*}S^{-1}U)
=
\begin{pmatrix}
A&0&0\\
0&B&0\\
0&0&C
\end{pmatrix}\sharp
\begin{pmatrix}
B^{-1}&0&0\\
0&C^{-1}&0\\
0&0&A^{-1}
\end{pmatrix}
=\begin{pmatrix}
A\sharp B^{-1}&0&0\\
0&B\sharp C^{-1}&0\\
0&0&C\sharp A^{-1}
\end{pmatrix}.
$$ 
Since $U^{3}=I$, 
\begin{align*}
&E_{U}(S\sharp U^{*}S^{-1}U)=
\dfrac{1}{3}\{S\sharp (U^{*}S^{-1}U)+
U^{*}\cdot S\sharp (U^{*}S^{-1}U)\cdot U+
{U^{*}}^{2}\cdot S\sharp (U^{*}S^{-1}U)\cdot U^{2}\}\\
&=\frac{1}{3}
\begin{pmatrix}
A\sharp B^{-1}+B\sharp C^{-1}+C\sharp A^{-1}&0&0\\
0&B\sharp C^{-1}+C\sharp A^{-1}+A\sharp B^{-1}&0\\
0&0&C\sharp A^{-1}+A\sharp B^{-1}+B\sharp C^{-1}
\end{pmatrix}.
\end{align*} 
Then using the assumption that Conjecture 3 is true, we get 
$$
\dfrac{A\sharp B^{-1}+B\sharp C^{-1}+C\sharp A^{-1}}{3}
\geq I. 
$$
\end{proof}

Therefore if we can find 
positive invertible matrices $A,B,C\in M_{n}({\Bbb C})$ which 
do not satisfy
$$
A\sharp B^{-1}+B\sharp C^{-1}+C\sharp A^{-1}\geq3I, 
\eqno{(\dagger)}
$$
we can conclude that Conjecture 1 
is not true in $M_{3n}({\Bbb C})$ and 
construct an explicit counter-example. 

Although we will construct a counter example to 
the conjecture in the next section, 
let us show that there are several evidences 
which support the validity of the conjecture. 
The following facts state that if we consider 
the trace in both sides of the inequalities, 
Conjecture 3 and the inequality $(\dagger)$ are true. 
\begin{prop}
\begin{enumerate} 
\item 
For any positive invertible matrix $S$ and 
any unitary matrix $U$ in $M_{n}({\Bbb C})$, we have 
$$
\dfrac{1}{n}{\rm Tr}(E_{U}(S\sharp (U^{*}S^{-1}U)))\geq 1.
$$

\item 
For any 
positive invertible matrices $A,B,C\in M_{n}({\Bbb C})$, 
we have 
$$
\dfrac{1}{n}{\rm Tr}(A\sharp B^{-1}+B\sharp C^{-1}+C\sharp A^{-1})
\geq3.
$$
\end{enumerate}
\end{prop}

\begin{proof}
For a positive invertible matrix $X\in M_{n}({\Bbb C})$ 
with eigenvalues $\{\lambda_{1},\cdots,\lambda_{n}\}$, 
we observe by concavity of the function $\log t$, 
$$
\dfrac{1}{n}\log{\rm det}(X)
=\dfrac{1}{n}(\log\lambda_{1}+\cdots+\log\lambda_{n})
\leq \log\dfrac{1}{n}(\lambda_{1}+\cdots+\lambda_{n})
=\log\dfrac{1}{n}{\rm Tr}(X)
$$ 
and hence 
$$
({\rm det}(X))^{\frac{1}{n}}\leq \dfrac{1}{n}{\rm Tr}(X).
$$ 

\noindent(i) 
$
\dfrac{1}{n}{\rm Tr}(E_{U}(S\sharp (U^{*}S^{-1}U)))
=\dfrac{1}{n}{\rm Tr}(S\sharp (U^{*}S^{-1}U))\geq 
({\rm det}(S\sharp (U^{*}S^{-1}U)))^{\frac{1}{n}}=1.
$

\noindent(ii) 
\begin{align*}
\dfrac{1}{n}{\rm Tr}&(A\sharp B^{-1}+B\sharp C^{-1}+C\sharp A^{-1})
=\dfrac{1}{n}{\rm Tr}(A\sharp B^{-1})+
\dfrac{1}{n}{\rm Tr}(B\sharp C^{-1})+\dfrac{1}{n}{\rm Tr}(C\sharp A^{-1})\\
&\geq ({\rm det}(A\sharp B^{-1}))^{\frac{1}{n}}+
({\rm det}(B\sharp C^{-1}))^{\frac{1}{n}}+
({\rm det}(C\sharp A^{-1}))^{\frac{1}{n}}\\
&={\rm det}(A)^{\frac{1}{2n}}{\rm det}(B)^{-\frac{1}{2n}}
+{\rm det}(B)^{\frac{1}{2n}}{\rm det}(C)^{-\frac{1}{2n}}
+{\rm det}(C)^{\frac{1}{2n}}{\rm det}(A)^{-\frac{1}{2n}}\\
&\geq 
3\{
{\rm det}(A)^{\frac{1}{2n}}{\rm det}(B)^{-\frac{1}{2n}}
\times{\rm det}(B)^{\frac{1}{2n}}{\rm det}(C)^{-\frac{1}{2n}}
\times{\rm det}(C)^{\frac{1}{2n}}{\rm det}(A)^{-\frac{1}{2n}}
\}^{\frac{1}{3}}=3.
\end{align*}
Here we used the usual arithmetic-geometric inequality 
$
\dfrac{a+b+c}{3}\geq(abc)^{\frac{1}{3}}
$. 
\end{proof}

By the jointly concavity of the geometric mean \cite{A1}, 
we see that 
$$
\Bigl(
\frac{A+B+C}{3}\Bigr)
\sharp
\Bigl(\frac{B^{-1}+C^{-1}+A^{-1}}{3}\Bigr)\geq 
\frac{1}{3}(A\sharp B^{-1}+B\sharp C^{-1}+C\sharp A^{-1}).
$$
Thus if the inequality $(\dagger)$ is true, we must have 
$$
\Bigl(
\frac{A+B+C}{3}\Bigr)
\sharp
\Bigl(\frac{B^{-1}+C^{-1}+A^{-1}}{3}\Bigr)\geq I. 
\eqno{(\ddagger)}
$$ 

\begin{prop}
For any 
positive invertible matrices $A,B,C\in M_{n}({\Bbb C})$, 
the inequality $(\ddagger)$ is true. 
\end{prop}

\begin{proof}
This is also a direct consequence from  the 
jointly concavity of the geometric mean. Indeed 
\begin{align*}
\Bigl(
\frac{A+B+C}{3}\Bigr)
\sharp
\Bigl(\frac{B^{-1}+C^{-1}+A^{-1}}{3}\Bigr)
&=\Bigl(
\frac{A+B+C}{3}\Bigr)
\sharp
\Bigl(\frac{A^{-1}+B^{-1}+C^{-1}}{3}\Bigr)\\
&\geq \frac{1}{3}(A\sharp A^{-1}+B\sharp B^{-1}+C\sharp C^{-1})
=3I.
\end{align*}
\end{proof}

Finally we would like to point out the following fact. 
For any 
positive invertible matrices $A,B\in M_{n}({\Bbb C})$, 
we can easily see that 
$$
A\sharp B^{-1}+B\sharp A^{-1}=
(A\sharp B^{-1})+(A\sharp B^{-1})^{-1}
\geq2.
$$

\section{a counter-example to the conjecture.} 

In this section we will construct a counter-example 
to Conjecture 1. 
This example is due to Professors Minghua Lin 
and Stephen Drury \cite{DL}. We would like to thank them. 

In the inequality 
$$
A\sharp B^{-1}+B\sharp C^{-1}+C\sharp A^{-1}\geq3I
$$ 
if we set $A=X^{2}$, $B=Y^{-2}$ and $C=I$ we obtain 
$$
X^{2}\sharp Y^{2}+X^{-1}+Y^{-1}\geq3I.
$$
We will show that there are 
two positive-definite matrices 
$X$ and $Y$ such that they do not satisfy this inequality. 
This means that there are $6\times6$ matrices 
which do not satisfy Conjecture 1.

The following fact is well-known for the specialists. 
We include its proof  for completeness. 
\begin{lem}
For $2\times2$ matrices $X>0$ and $Y>0$, we have 
$$
X\sharp Y=
\dfrac{({\rm det}(X){\rm det}(Y))^{\frac{1}{4}}}
{{\rm det}\Bigl(\frac{1}{\sqrt{{\rm det}(X)}}X+
\frac{1}{\sqrt{{\rm det}(Y)}}Y
\Bigr)^{\frac{1}{2}}}
\Bigl(\frac{1}{\sqrt{{\rm det}(X)}}X+
\frac{1}{\sqrt{{\rm det}(Y)}}Y\Bigr).
$$ 
In particular if
${\rm det}(X)={\rm det}(Y)$, we have 
$$
X\sharp Y=\sqrt{
\dfrac{{\rm det}(X)}
{{\rm det}(X+Y)}}
(X+Y).
$$

\end{lem}

\begin{proof}
Applying the Cayley-Hamilton theorem to the matrix 
$(X^{-\frac{1}{2}}YX^{-\frac{1}{2}})^{\frac{1}{2}}$ 
we have 
$$
X^{-\frac{1}{2}}YX^{-\frac{1}{2}}-
{\rm Tr}((X^{-\frac{1}{2}}YX^{-\frac{1}{2}})^{\frac{1}{2}})
(X^{-\frac{1}{2}}YX^{-\frac{1}{2}})^{\frac{1}{2}}+
\Bigl(\frac{{\rm det}(Y)}{{\rm det}(X)}\Bigr)^{\frac{1}{2}}=0.
$$
By multiplying $X^{\frac{1}{2}}$ from both sides we see that 
$$
Y-{\rm Tr}((X^{-\frac{1}{2}}YX^{-\frac{1}{2}})^{\frac{1}{2}})
X\sharp Y+
\Bigl(\frac{{\rm det}(Y)}{{\rm det}(X)}\Bigr)^{\frac{1}{2}}X=0. 
$$
Hence we can write 
$$
X\sharp Y=
c
\Bigl(\frac{1}{\sqrt{{\rm det}(X)}}X+
\frac{1}{\sqrt{{\rm det}(Y)}}Y\Bigr).
$$
By taking the determinants we have 
$$
({\rm det}(X){\rm det}(Y))^{\frac{1}{2}}=c^{2}
{\rm det}\Bigl(\dfrac{1}{\sqrt{{\rm det}(X)}}X+
\frac{1}{\sqrt{{\rm det}(Y)}}Y\Bigr).
$$
So we are done. 
\end{proof}

Set 
$$
X=\dfrac{1}{5^{2}}
\begin{pmatrix}
50&5\\
5&1
\end{pmatrix},\ \ 
Y=\dfrac{1}{5^{2}}
\begin{pmatrix}
50&-5\\
-5&1
\end{pmatrix},\ \ 
P=
\begin{pmatrix}
1&0\\
0&0
\end{pmatrix}
.
$$
Here we remark that 
${\rm det}(X)={\rm det}(Y)=\dfrac{1}{5^{2}}$ and 
$$
X^{2}=\dfrac{1}{5^{4}}
\begin{pmatrix}
2525&255\\
255&26
\end{pmatrix},\ \ 
Y^{2}=\dfrac{1}{5^{4}}
\begin{pmatrix}
2525&-255\\
-255&26
\end{pmatrix}.
$$
By the previous lemma we know that 
$$
X^{2}\sharp Y^{2}=\sqrt{
\dfrac{{\rm det}(X^{2})}
{{\rm det}(X^{2}+Y^{2})}}
(X^{2}+Y^{2}).
$$ 
Since 
$
X^{2}+Y^{2}=
\dfrac{1}{5^{4}}
\begin{pmatrix}
5050&0\\
0&52
\end{pmatrix}
$, we compute 
$$
P(X^{2}\sharp Y^{2})P
=\dfrac{\frac{1}{5^{2}}}
{\frac{1}{5^{4}}(5050\times52)^{\frac{1}{2}}}\times 
\dfrac{5050}{5^{4}}P
=
\sqrt{\dfrac{101}{650}}P.
$$
Since 
$$
X^{-1}=
\begin{pmatrix}
1&-5\\
-5&50
\end{pmatrix},\ \ 
Y^{-1}=
\begin{pmatrix}
1&5\\
5&50
\end{pmatrix},
$$
we see that 
$$
P(X^{2}\sharp Y^{2}+X^{-1}+Y^{-1})P
=\sqrt{\dfrac{101}{650}}P
+2P<3P.
$$
Therefore we conclude that the matrices $X$ and 
$Y$ do not satisfy the inequality 
$$
X^{2}\sharp Y^{2}+X^{-1}+Y^{-1}\geq3I.
$$

\section{the conjecture for $2\times2$ matrices} 
In \cite{A2} Ando showed that Conjecture 1 is true for 
$2\times2$-matrices. In this section we will give another 
proof for this result. 
In section 2 we saw that Conjecture 1 is equivalent to 
Conjecture 3. 
Thus it is enough to show the following. 

\begin{thm}
For any positive invertible $2\times2$ matrix $S$ and 
any unitary $2\times2$ matrix $U$, we have 
$$
E_{U}(S\sharp (U^{*}S^{-1}U))\geq I.
$$
\end{thm}

\begin{proof} 
Without loss of generality we may 
assume that $U$ is a diagonal matrix of the form 
$U=
\begin{pmatrix}
1&0\\
0&z
\end{pmatrix}
$ with $|z|=1$ because 
$(wU)^{*}S^{-1}(wU)=U^{*}S^{-1}U$ 
for any complex number $w$ with $|w|=1$. 
In the case that $z=1$, $U$ becomes identity 
and so the statement is obvious. 
Therefore we have only to consider the case 
$z\not=1$ and $U\not=I$. 
Here we remark that in this case the map 
$E_{U}$ is defined by 
$$
E_{U}(
\begin{pmatrix}
x&y\\
z&w
\end{pmatrix}
)=
\begin{pmatrix}
x&0\\
0&w
\end{pmatrix}.
$$ 
We can also assume that 
$
S=
\begin{pmatrix}
a&b\\
\overline{b}&c
\end{pmatrix}
$ 
with 
$
{\rm det}(S)=ac-|b|^{2}=1
$ since 
$
(\alpha S)\sharp \{U^{*}(\alpha S)^{-1}U\}
=S\sharp (U^{*}S^{-1}U)
$ 
for any positive number $\alpha$. Then we see that 
$$
S^{-1}=
\begin{pmatrix}
c&-b\\
-\overline{b}&a
\end{pmatrix},\ \ 
U^{*}S^{-1}U=
\begin{pmatrix}
c&-bz\\
-\overline{bz}&a
\end{pmatrix},\ \ 
S+U^{*}S^{-1}U=
\begin{pmatrix}
a+c&b(1-z)\\
\overline{b(1-z)}&a+c
\end{pmatrix}.
$$
Then we compute 
$$
{\rm det}(S+U^{*}S^{-1}U)
=(a+c)^{2}-|b(1-z)|^{2}
=2(ac-|b|^{2})+a^{2}+c^{2}+
2|b|^{2}{\rm Re}z=a^{2}+c^{2}+
2(1+|b|^{2}{\rm Re}z).
$$
Then since ${\rm det}(S)={\rm det}(U^{*}S^{-1}U)=1$, 
by lemma 3.1 we have 
\begin{align*}
S\sharp U^{*}S^{-1}U&=
\sqrt{
\dfrac{{\rm det}(S)}
{{\rm det}(S+U^{*}S^{-1}U)}}
(S+U^{*}S^{-1}U)\\
&=\dfrac{1}{\sqrt{a^{2}+c^{2}+
2(1+|b|^{2}{\rm Re}z)}}
\begin{pmatrix}
a+c&b(1-z)\\
\overline{b(1-z)}&a+c
\end{pmatrix}
\end{align*} 
and hence 
$$
E_{U}(S\sharp U^{*}S^{-1}U)=
\dfrac{1}{\sqrt{a^{2}+c^{2}+
2(1+|b|^{2}{\rm Re}z)}}
\begin{pmatrix}
a+c&0\\
0&a+c
\end{pmatrix}.
$$
On the other hand we see that  
\begin{align*}
(a+c)^{2}&-\{a^{2}+c^{2}+
2(1+|b|^{2}{\rm Re}z)\}=
2\{(ac-1)-|b|^{2}{\rm Re}z\}\\
&=2(|b|^{2}-|b|^{2}{\rm Re}z)
=2|b|^{2}(1-{\rm Re}z)\geq0.
\end{align*}
Here we used the fact that $ac-1=|b|^{2}$. 
Therefore we conclude 
$$
E_{U}(S\sharp U^{*}S^{-1}U)\geq I.
$$
\end{proof}


\begin{thebibliography}{99}
\bibitem{A1} 
T.~Ando, {\it Concavity of certain maps on positive definite matrices and applications to Hadamard products.} Linear Algebra Appl. {\bf 26} (1979), 203--241. 


\bibitem{A2} 
\underline{\phantom{aaaaa}}, {\it 
Geometric mean and norm Schwarz inequality.} 
Ann. Funct. Anal. {\bf 7} (2016), no. 1, 1--8. 

\bibitem{DL} S.~Drury and M.~Lin, private communications. 

\end{thebibliography}
\end{document}